\documentclass[10pt]{amsart}
\usepackage{amsmath,amssymb,xy,array}
\usepackage[T1]{fontenc}
\usepackage{amsfonts}
\usepackage{amsmath}
\usepackage{amssymb}
\usepackage{amsthm}
\usepackage{hyperref}
\usepackage{graphicx}
\usepackage{fancyhdr}
\usepackage{tikz}
\usepackage{longtable}
\usepackage{geometry}
\usepackage{textcomp}
\usepackage[utf8]{inputenc}
\usetikzlibrary{trees}

\providecommand{\U}[1]{\protect\rule{.1in}{.1in}}
\providecommand{\U}[1]{\protect\rule{.1in}{.1in}}
\providecommand{\U}[1]{\protect\rule{.1in}{.1in}}
\providecommand{\U}[1]{\protect\rule{.1in}{.1in}}
\providecommand{\U}[1]{\protect\rule{.1in}{.1in}}
\input{xy}
\xyoption{all}
\usepackage{amsfonts}
\usepackage{amssymb,amsthm}
\usepackage{verbatim}
\usepackage{hyperref} 
\usepackage[T1]{fontenc}
\usepackage{amsmath}
\usepackage{hyperref}
\usepackage{enumerate}
\usepackage{tikz}
\usepackage{color}

\newcommand{\G}{\mathbb{G}}

\newcommand{\C}{\mathbb C}
\newcommand{\p}{\mathbb P}

\DeclareMathOperator{\Sing}{Sing}

\DeclareMathOperator{\Sec}{Sec}

\newcommand{\QED}{\ifhmode\unskip\nobreak\fi\quad {\rm Q.E.D.}} 

\renewcommand{\sec}{\mathbb{S}ec}
\DeclareMathOperator{\expdim}{expdim}

\newtheorem{thm}{Theorem}[section]
\newtheorem{Lemma}[thm]{Lemma}
\newtheorem{Proposition}[thm]{Proposition}
\newtheorem{Corollary}[thm]{Corollary}

\newtheorem*{thm*}{Theorem}

\theoremstyle{definition}

\newtheorem{Definition}[thm]{Definition}
\newtheorem{Remark}[thm]{Remark}

\begin{document}
\title{On tangential weak defectiveness and identifiability of projective varieties}

\author[Ageu Barbosa Freire]{Ageu Barbosa Freire}
\address{\sc Ageu Barbosa Freire\\
Instituto de Matem\'atica e Estat\'istica, Universidade Federal Fluminense, Campus Gragoat\'a, Rua Alexandre Moura 8 - S\~ao Domingos\\
24210-200 Niter\'oi, Rio de Janeiro\\ Brazil}
\email{ageufreire@id.uff.br}

\author[Alex Casarotti]{Alex Casarotti}
\address{\sc Alex Casarotti\\ Dipartimento di Matematica e Informatica, Universit\`a di Ferrara, Via Machiavelli 30, 44121 Ferrara, Italy}
\email{csrlxa@unife.it}

\author[Alex Massarenti]{Alex Massarenti}
\address{\sc Alex Massarenti\\ Dipartimento di Matematica e Informatica, Universit\`a di Ferrara, Via Machiavelli 30, 44121 Ferrara, Italy}
\email{alex.massarenti@unife.it}

\date{\today}
\subjclass[2020]{Primary 14N07; Secondary 14N05, 14N15, 14M15, 15A69, 15A75}
\keywords{Secant varieties, secant defectiveness, weak defectiveness, tangential weak defectiveness, identifiability}

\begin{abstract}
A point $p\in\mathbb{P}^N$ of a projective space is $h$-identifiable, with respect to a variety $X\subset\mathbb{P}^N$, if it can be written as linear combination
of $h$ elements of $X$ in a unique way. Identifiability is implied by conditions on the contact locus in $X$ of general linear spaces called non weak defectiveness and non tangential weak defectiveness. We give conditions ensuring non tangential weak defectiveness of an irreducible and non-degenerated projective variety $X\subset\mathbb{P}^N$, and we apply these results to Segre-Veronese varieties.   
\end{abstract}

\maketitle
\setcounter{tocdepth}{1}
\tableofcontents

\section{Introduction}
A point $p\in\mathbb{P}^N$ of a projective space is \textit{$h$-identifiable} with respect to a variety $X\subset\mathbb{P}^N$ if it can be written as linear combination
of $h$ elements of $X$ in a unique way.

Identifiability problems and techniques are of relevance in both pure and applied mathematics. For instance, identifiability algorithms have applications in psycho-metrics, chemometrics, signal processing, numerical linear algebra, computer vision, numerical analysis, neuroscience and graph analysis \cite{BK09}, \cite{CM96}, \cite{CGLM08}. In pure mathematics identifiability questions often appears in rationality problems \cite{MM13}, \cite{Ma16}.
 
Identifiability has been related to the concept of weak defectiveness in \cite{Me06}, and more recently to the notion of tangential weak defectiveness in \cite{CO12}. 
 
We introduce the concept of \textit{$(h,s)$-tangential weakly defectiveness}, where $h,s$ are positive integers. A variety $X\subset\mathbb{P}^N$ is $(h,s)$-tangentially weakly defective if a general linear subspace of dimension $s$, which is tangent to $X$ at $h$ general points $x_1,\dots,x_h\in X$, is tangent to $X$ along a positive dimensional subvariety of $X$ containing at least one of the $x_i$. In particular, when $s = \dim\left\langle T_{x_1}X,\dots, T_{x_h}X\right\rangle$ we recover the notion of \textit{$h$-tangential weak defectiveness} while for $s = N-1$ we get the notion of \textit{$h$-weak defectiveness}. 

The \textit{$h$-secant variety} $\mathbb{S}ec_{h}(X)$ of a non-degenerate $n$-dimensional variety $X\subset\mathbb{P}^N$ is the Zariski closure of the union of all linear spaces spanned by collections of $h$ points of $X$. The \textit{expected dimension} of $\mathbb{S}ec_{h}(X)$ is $\expdim(\mathbb{S}ec_{h}(X)):= \min\{nh+h-1,N\}$. The actual dimension of $\mathbb{S}ec_{h}(X)$ may be smaller than the expected one. Following \cite[Section 2]{CC10}, we say that $X$ is \textit{$h$-defective} if $\dim(\mathbb{S}ec_{h}(X)) < \expdim(\mathbb{S}ec_{h}(X))$.

Note that if $X\subset\mathbb{P}^N$ is $(h,s)$-tangentially weakly defective then it is $(h,s')$-tangentially weakly defective for any $s'\geq s$. Furthermore, if $X\subset\mathbb{P}^N$ is $h$-defective then it is $(h,s)$-tangentially weakly defective for all $s\geq \dim\left\langle T_{x_1}X,\dots, T_{x_h}X\right\rangle$. Moreover, if $X\subset\mathbb{P}^N$ is not $h$-tangentially weakly defective then it is $h$-identifiable. In Section \ref{sec1} we recall all these notions and the relations among them in detail. 

In Section \ref{or-wd}, mixing the notion of osculating regularity introduced in \cite{MR19} with that of weak defectiveness, we prove a general result for producing bounds yielding the non $(h,s)$-tangential weak defectiveness of a projective variety $X\subset\mathbb{P}^N$. Thanks to this machinery in Section \ref{wd-SV} we prove a number of results on weak defectiveness of Segre-Veronese varieties. Given two $r$-uples $\textbf{\textit{n}}=(n_1,\dots,n_r)$ and $\textbf{\textit{d}} = (d_1,\dots,d_r)$ of positive integers, with $n_1\leq \dots \leq n_r$ we will denote by $SV^{\textbf{\textit{n}}}_{\textbf{\textit{d}}}\subset\mathbb{P}^N$ the corresponding Segre-Veronese variety that is the product  $\mathbb{P}^{n_1}\times\dots\times\mathbb{P}^{n_r}$ embedded by the complete 
linear system  $\big|\mathcal{O}_{\mathbb{P}^{n_1}\times\dots\times\mathbb{P}^{n_r}}(d_1,\dots, d_r)\big|$. Our main results in Propositions \ref{Prop1}, \ref{Prop2}, \ref{Prop0}, \ref{Prop1bis}, \ref{Prop3bis}, \ref{Prop4bis}, Theorems \ref{Bound_non_Wd}, \ref{th1wd} and Remark \ref{asy_wd} can be summarized as follows.

\begin{thm}\label{main1}
If $h\leq (n_1+1)^{\lfloor\log_2(d)\rfloor}$ then the Segre-Veronese variety $SV^{\textbf{\textit{n}}}_{\textbf{\textit{d}}}\subset\mathbb{P}^N$ is not $h$-weakly defective, where $d = \min\{d_1,\ldots,d_r\}$. In particular, under this bound $SV^{\textbf{\textit{n}}}_{\textbf{\textit{d}}}\subset\mathbb{P}^N$ is not $h$-defective. Furthermore, $SV^{\textbf{\textit{n}}}_{\textbf{\textit{d}}}$ is $1$-weakly defective if and only if $d_r = 1$ and $n_r > \sum_{i=1}^{r-1}n_i$. 

Moreover, consider $SV_{\textbf{\textit{d}}}^{\textbf{\textit{n}}}$ with $\textbf{\textit{n}}=(n_{1},\dots,n_{r})$ and $\textbf{\textit{d}}=(d_{1},\dots,d_{r-1},1)$, and assume that $n_{r}> \sum_{i=1}^{r-1} n_{i}$. If 
$$s\leq \prod_{i=2}^r{\binom{n_i+d_i}{n_i}}-n_{r}\sum_{i=1}^{r-1}n_{i}$$ 
then $SV_{\textbf{\textit{d}}}^{\textbf{\textit{n}}}$ is not $(1,s)$-tangentially weakly defective.

Finally, if $\textbf{\textit{n}}=(1,n)$ and $\textbf{\textit{d}}=(1,d)$ then $SV_{\textbf{\textit{d}}}^{\textbf{\textit{n}}}$ is not $(1,s)$-tangentially weakly defective if and only if $s\leq d(n+1)$.
\end{thm}

In Section \ref{wdfib} we give a criterion for non tangential weak defectiveness of products, and we apply it to Segre-Veronese varieties. Our main result is the following:

\begin{thm}\label{main2}
Consider a Segre-Veronese variety $SV_{\textbf{\textit{d}}}^{\textbf{\textit{n}}}\subset\p^{N(\textbf{\textit{n}},\textbf{\textit{d}})}$ with $\textbf{\textit{n}}=(1,n_2,\dots,n_r)$ and $\textbf{\textit{d}}=(1,d_2,\dots,d_r)$.
Assume that $n_2 \leq n_3 \leq \dots \leq n_r$ and let $d:=\min\{d_i\}-1$. If $$h < h_{n_2+1}(d) \sim n_2^{\lfloor log_2(d) \rfloor}$$  
then $SV_{\textbf{\textit{d}}}^{\textbf{\textit{n}}}$ is not $h$-tangentially weakly defective, and hence $SV_{\textbf{\textit{d}}}^{\textbf{\textit{n}}}$ is $h$-identifiable. In particular, under this bound $SV_{\textbf{\textit{d}}}^{\textbf{\textit{n}}}$ is not $h$-defective. 
\end{thm} 

We would like to stress that, as noticed in Remark \ref{sd-nn}, the non secant defectiveness of $SV_{\textbf{\textit{d}}}^{\textbf{\textit{n}}}$ is not needed in the proof of Theorem \ref{main2}. For results and conjectures on the secant dimensions of Segre-Veronese varieties we refer to \cite{AB12}, \cite{AB13}, \cite{AB09}, \cite{LP13} and \cite{AMR17}. Finally, we would like to mention that results on the identifiablity of $SV_{\textbf{\textit{d}}}^{\textbf{\textit{n}}}$, under hypotheses on its non secant defectiveness, have been recently given in \cite{BBC18}.   

\subsection*{Acknowledgments}
The first named author would like to thank FAPERJ and Massimiliano Mella (PRIN $2015$, Geometry of Algebraic Varieties, 2015EYPTSB-005) for the financial support, and the University of Ferrara for the hospitality during the period in which the majority of this work was completed. We thank Luca Chiantini for pointing out a mistake in a previous version of Theorem \ref{gen-twd} and Corollary \ref{CorSV}.

The third named author is a member of the Gruppo Nazionale per le Strutture Algebriche, Geometriche e le loro Applicazioni of the Istituto Nazionale di Alta Matematica F. Severi (GNSAGA-INDAM). We thank the referee for the helpful comments that helped us to improve the paper.

\section{Secant defectiveness, $(h,s)$-tangential weak defectiveness and identifiability}\label{sec1}
Throughout the paper we work over the field of complex numbers. In this section we recall the notions of secant variety, secant defectiveness and identifiability. We refer to \cite{Ru03} for a nice and comprehensive survey on the subject.

Let $X\subset\mathbb{P}^N$ be an irreducible non-degenerate variety of dimension $n$ and let $\Gamma_h(X)\subset X\times \dots \times X\times\G(h-1,N)$, where $h\leq N$, be the closure of the graph of the rational map $\alpha: X\times\dots\times X \dasharrow \G(h-1,N)$ taking $h$ general points to their linear span $\langle x_1, \dots , x_{h}\rangle$. Observe that $\Gamma_h(X)$ is irreducible and reduced of dimension $hn$. Let $\pi_2:\Gamma_h(X)\to\G(h-1,N)$ be the natural projection, and $\mathcal{S}_h(X):=\pi_2(\Gamma_h(X))\subset\G(h-1,N)$. Again $\mathcal{S}_h(X)$ is irreducible and reduced of dimension $\min\{hn,h(N-h+1)\}$. Finally, let
$$\mathcal{I}_h=\{(x,\Lambda) \: | \: x\in \Lambda\}\subset\mathbb{P}^N\times\G(h-1,N)$$
with natural projections $\pi_h$ and $\psi_h$ onto the factors. The \textit{abstract $h$-secant variety} is the irreducible variety
$$\Sec_{h}(X):=(\psi_h)^{-1}(\mathcal{S}_h(X))\subset \mathcal{I}_h$$
The \textit{$h$-secant variety} is defined as
$$\sec_{h}(X):=\pi_h(Sec_{h}(X))\subset\mathbb{P}^N$$
It immediately follows that $\Sec_{h}(X)$ is an $(hn+h-1)$-dimensional variety with a $\mathbb{P}^{h-1}$-bundle structure over $\mathcal{S}_h(X)$. We say that $X$ is \textit{$h$-defective} if $\dim\sec_{h}(X)<\min\{\dim\Sec_{h}(X),N\}$.

Now, let $X^{(h)}$ be the symmetric product of $h$-copies of $X$, and consider the locus $S^X_h\subset X^{(h)}$ parametrizing sets of distinct points. Given a point $y\in S^X_h$, corresponding to $h$ distinct points $x_1,\dots,x_h\in X$, we will denote by $\left\langle y\right\rangle$ the linear span $\left\langle x_1,\dots,x_h\right\rangle\subset\mathbb{P}^N$.

\begin{Definition}
A point $p\in\mathbb{P}^N$ has rank $h$ with respect to $X$ if $p\in \left\langle y\right\rangle$ for some $y\in S^X_h$ but $p\notin \left\langle y\right\rangle$ for all $y\in S^X_{k}$ for any $k < h$.

A point $p\in\mathbb{P}^N$ is $h$-identifiable with respect to $X$ if $p$ has rank $h$ with respect to $X$ and $(\pi_h)^{-1}(p)$ is a single point. The variety $X$ is $h$-identifiable if the general point of $\sec_h(X)$ is $h$-identifiable. 
\end{Definition}   

Note that by Terracini's lemma \cite{Te11} if $y\in\sec_h(X)$ is a general point lying in the span of $x_1,\dots,x_h\in X$ then $T_y\sec_h(X) = \left\langle T_{x_1}X,\dots,T_{x_h}X\right\rangle$. Therefore, if $X$ is $h$-defective then the general hyperplane tangent to $X$ at $h$ points is tangent to $X$ along a positive dimensional subvariety. 

\begin{Definition}
Let $x_1,\dots,x_h\in X$ be general points, and let $H$ be a hyperplane tangent to $X$ at $x_1,\ldots,x_h$. The $h$-contact locus $\Sigma_{x_1,\ldots,x_h,H}$ of $X$ with respect to $x_1,\dots,x_h,H$ is defined as the union of the irreducible components of $\Sing(X\cap H)$ containing at least one of the $x_i$. Now, $X$ is said to be $h$-weakly defective if $\Sigma_{x_1,\ldots,x_h,H}$ has positive dimension for $H$ a general hyperplane containing $\left\langle T_{x_1}X,\dots,T_{x_h}X\right\rangle$.  
\end{Definition} 

Therefore, if $X$ is $h$-defective then it is $h$-weakly defective. However, the converse does not hold in general. For instance, if we denote by $V^n_d\subset\mathbb{P}^N$ the degree $d$ Veronese embedding of $\mathbb{P}^n$ we have that for $(d,n)\in\{(6,2),(4,3),(3,5)\}$ the Veronese $V_d^n$ is never defective but it is respectively $9$-weakly defective, $8$-weakly defective and $9$-weakly defective \cite{CC02}. 

Furthermore, by the infinitesimal Bertini's theorem \cite[Theorem 1.4]{CC02} if $X$ is not $h$-weakly defective then it is $h$-identifiable. Recently, a result translating non secant defectiveness into identifiability has been proven in \cite{CM19}.

\begin{Definition}
Let $x_1,\dots,x_h\in X$ be general points. The $h$-tangential contact locus $\Gamma_{x_1,\dots,x_h}$ of $X$ with respect to $x_1,\dots,x_h$ is the closure in $X$ of
the union of all the irreducible components which contain at least one of the $x_i$ of the locus of points of $X$ where $\left\langle T_{x_1}X,\dots, T_{x_h}X\right\rangle$ is tangent to $X$. Let $\gamma_{x_1,\dots,x_h}$ be the largest dimension of the components of $\Gamma_{x_1,\dots,x_h}$. If $\gamma_{x_1,\dots,x_h} > 0$ we say that $X$ is $h$-tangentially weakly defective.
\end{Definition}

Clearly, if $X$ is $h$-tangentially weakly defective then it is $h$-weakly defective. Moreover, by \cite[Proposition 2.4]{CO12} if $X$ is not $h$-tangentially weakly defective then it is $h$-identifiable. However, the Grassmannian $\mathbb{G}(2,7)$ parametrizing planes in $\mathbb{P}^7$ is $3$-tangentially weakly defective but it is $3$-identifiable \cite[Proposition 1.7]{BV18}.

Finally, we introduce a notion that measures how much a $h$-weakly defective variety is far from being $h$-tangentially weakly defective. 

\begin{Definition}\label{gen_def}
Let $x_1,\dots,x_h\in X$ be general points and $\Pi\subset\mathbb{P}^N$ a linear subspace of dimension $s$ containing $\left\langle T_{x_1}X,\dots, T_{x_h}X\right\rangle$. The $(h,s)$-tangential contact locus $\Gamma_{x_1,\dots,x_h,\Pi}$ of $X$ with respect to $x_1,\dots,x_h, \Pi$ is the closure in $X$ of
the union of all the irreducible components which contain at least one of the $x_i$ of the locus of points of $X$ where $\Pi$ is tangent to $X$. Let $\gamma_{x_1,\dots,x_h,\Pi}$ be the largest dimension of the components of $\Gamma_{x_1,\dots,x_h,\Pi}$. If $\gamma_{x_1,\dots,x_h,\Pi} > 0$ for $\Pi$ general, we say that $X$ is $(h,s)$-tangentially weakly defective.
\end{Definition}

In particular, when $s = \dim\left\langle T_{x_1}X,\dots, T_{x_h}X\right\rangle$ from Definition \ref{gen_def} we recover the notion of $h$-tangential weak defectiveness while for $s = N-1$ we get the notion of $h$-weak defectiveness.

\section{Osculating regularity and weak defectiveness}\label{or-wd}
We begin by proving a simple result on the behavior of contact loci under flat degenerations.

\begin{Lemma}\label{lemma_sc}
Let $X\subset \p^N$ be a projective variety, $\Delta\subset\mathbb{C}$ a complex disk around the origin and $\{\Pi_t\}_{t\in\Delta}$ a family of linear subspaces of $\mathbb{P}^N$. Then
$$\dim(\Sing(\Pi_0\cap X))\geq \dim(\Sing(\Pi_t\cap X))$$
for $t\in \Delta$. 

Furthermore, let $\{\Gamma_t\}_{t\in\Delta}$ be a family of linear subspaces $\Gamma_t\subset\mathbb{P}^N$, $\Lambda\subset\mathbb{P}^N$ a linear subspace containing $\Gamma_0$, and $\Pi$ a linear subspace containing $\Lambda$. Then 
$$\dim(\Sing(\widetilde{\Pi}_t\cap X))\leq \dim(\Sing(\Pi\cap X))$$
where $\widetilde{\Pi}_t$ is a general linear subspace of dimension $\dim(\Pi)$ containing $\Gamma_t$.
\end{Lemma}
\begin{proof}
For the first claim it is enough to consider the variety
$$Y = \{(x,t)\: | \: x\in\Sing(X\cap \Pi_t)\}\subset X\times\Delta$$
with projection $\pi_2:Y\rightarrow \Delta$ and to conclude by semi-continuity.

For the second part note that since $\Gamma_0\subseteq\Lambda$ we have that $\Gamma_0\subseteq \Pi$. Let $\Gamma'\subset \Pi$ be a subspace such that $\Pi = \left\langle \Gamma_0, \Gamma'\right\rangle$, $\Gamma'\cap\Gamma_0 = \emptyset$, and set $\Pi_t = \left\langle \Gamma_t, \Gamma'\right\rangle$. Then $\{\Pi_t\}_{t\in\Delta}$ is a family of linear subspace such that $\Gamma_t\subset \Pi_t$ for all $t\in\Delta$. By the first part of the proof we have $\dim(\Sing(\Pi\cap X))\geq \dim(\Sing(\Pi_t\cap X))$ for all $t\in\Delta$. Now, consider the Grassmannian $\mathbb{G}(\dim(\Pi)-\dim(\Gamma_t)-1,N-\dim(\Gamma_t)-1)$ parametrizing $\dim(\Pi)$-dimensional linear subspaces of $\mathbb{P}^N$ containing $\Gamma_t$, and the variety
$$Z = \{(x,\widetilde{\Pi}_t) \: | \: x\in \Sing(\widetilde{\Pi}_t\cap X)\}\subseteq X\times \mathbb{G}(\dim(\Pi)-\dim(\Gamma_t)-1,N-\dim(\Gamma_t)-1)$$
with projection $\pi_2:Z\rightarrow \mathbb{G}(\dim(\Pi)-\dim(\Gamma_t)-1,N-\dim(\Gamma_t)-1)$. Again by semi-continuity we have
$$\dim(\Sing(\widetilde{\Pi}_t\cap X))\leq \dim(\Sing(\Pi_t\cap X))$$
for $\widetilde{\Pi}_t\in \mathbb{G}(\dim(\Pi)-\dim(\Gamma_t)-1,N-\dim(\Gamma_t)-1)$ general, and hence $\dim(\Sing(\Pi\cap X))\geq \dim(\Sing(\Pi_t\cap X))\geq \dim(\Sing(\widetilde{\Pi}_t\cap X))$.  
\end{proof}

Let $X\subset \p^N$ be a projective variety of dimension $n$, $p\in X$ a smooth point, and 
$$
\begin{array}{cccc}
\phi: &\mathcal{U}\subseteq\mathbb{C}^n& \longrightarrow & \mathbb{C}^{N}\\
      & (t_1,\dots,t_n) & \mapsto & \phi(t_1,\dots,t_n)
\end{array}
$$
with $\phi(0)=p$, a local parametrization of $X$ in a neighborhood of $p\in X$. 

For any $s\geq 0$ let $O^s_pX$ be the affine subspace of $\mathbb{C}^{N}$ passing through $p\in X$, and whose direction is given by the subspace generated by the vectors $\phi_I(0)$, where $I = (i_1,\dots,i_r)$ is a multi-index such that $|I|\leq s$ and $\phi_I = \frac{\partial^{|I|}\phi}{\partial t_1^{i_1}\dots\partial t_r^{i_r}}$.

\begin{Definition}\label{oscdef}
The $s$-\textit{osculating space} $T_p^s X$ of $X$ at $p$ is the projective closure in $\mathbb{P}^N$ of the affine subspace $O^s_pX\subseteq \mathbb{C}^{N}$.
\end{Definition}

For instance, $T_p^0 X=\{p\}$, and $T_p^1 X$ is the usual tangent space of $X$ at $p$. When no confusion arises we will write $T_p^s$ instead of $T_p^sX$. Now, let us recall \cite[Definition 5.5, Assumption 5.2]{MR19} and \cite[Definition 4.4]{AMR17}.

\begin{Definition}\label{osc_reg}
Let $X\subset\mathbb{P}^N$ be a projective variety. We say that $X$ has \textit{$m$-osculating regularity} if the following property holds: given general points $p_1,\dots,p_{m}\in X$ and an integer $s\geq 0$, 
there exists a smooth curve $C$ and morphisms $\gamma_j:C\to X$, $j=2,\dots,m$, 
such that  $\gamma_j(t_0)=p_1$, $\gamma_j(t_\infty)=p_j$, and the flat limit $T_0$ in the Grassmannian of the family of linear spaces 
$$
T_t=\left\langle T^{s}_{p_1},T^{s}_{\gamma_2(t)},\dots,T^{s}_{\gamma_{m}(t)}\right\rangle,\: t\in C\backslash \{t_0\}
$$
is contained in $T^{2s+1}_{p_1}$. 

We say that $X$ has \textit{strong $2$-osculating regularity} if the following property holds: given general points $p,q\in X$ and  integers $s_1,s_2\geq 0$, there exists a smooth curve $\gamma:C\to X$ such that $\gamma(t_0)=p$, $\gamma(t_\infty)=q$ and the flat limit $T_0$ in the Grassmannian of the family of linear spaces 
$$
T_t=\left\langle T^{s_1}_p,T^{s_2}_{\gamma(t)}\right\rangle,\: t\in C\backslash \{t_0\}
$$
is contained in $T^{s_1+s_2+1}_p$.
\end{Definition}
For a discussion on the notions of $m$-osculating regularity and strong $2$-osculating regularity and their application to Grassmannians, Segre-Veronese varieties, Lagrangian Grassmannians and Spinor varieties, and flag varieties we refer to \cite{MR19}, \cite{AMR17}, \cite{FMR20}, \cite{FCM19}.

Now, we define a function $h_m:\mathbb{N}_{\geq0}\longrightarrow\mathbb{N}_{\geq 0}$ counting how many tangent spaces can be degenerated into a higher order osculating space.

\begin{Definition}\label{h_m}
Given an integer $m\geq 0$ we define a function
$$h_m:\mathbb{N}_{\geq0}\longrightarrow\mathbb{N}_{\geq0}$$
as follows: $h_m(0)=0$ and for any $k>0$ write
$$k+1=2^{\lambda_1}+2^{\lambda_2}+\cdots+2^{\lambda_a}+\varepsilon$$
where $\lambda_1>\lambda_2>\cdots>\lambda_a\geq 1$ and $\varepsilon\in \{0,1\}$, then
$$h_m(k)=m^{\lambda_1-1}+m^{\lambda_2-1}+\cdots+m^{\lambda_a-1}$$
\end{Definition}

We are ready to prove the main result of this section relating osculating regularity to tangential weak defectiveness. 

\begin{thm}\label{Non_Weakly_defec}
Let $X\subset \p^N$ be a projective variety having $m$-osculating regularity and strong $2$-osculating regularity. Assume that there exist integers $l,k_1,\ldots,k_l\geq 1$, general points $p_1,\ldots,p_l\in X$ and a linear subspace of dimension $s$ containing $\langle T_{p_1}^{k_1},\ldots,T_{p_l}^{k_l}\rangle$ that is not tangent to $X$ along a positive dimensional subvariety. Set
$$h:=\sum_{j=1}^lh_m(k_j)$$
Then $X$ is not $(h,s)$-tangentially weakly defective.
\end{thm}
\begin{proof}
Let us consider the linear span 
$$T = \left\langle T^1_{p_1^1},\dots, T^1_{p_1^{h_m(k_1)}},\dots, T^1_{p_l^1},\dots,T^1_{p_l^{h_m(k_l)}}\right\rangle$$
and $p_1^1 = p_1,\dots, p_l^1 = p_l$. For seek of notational simplicity along the proof we will assume $l = 1$. For the general case it is enough to apply the same argument $l$ times.

Let us begin with the case $k_1+1 = 2^{\lambda}$. Then $h_{m}(k_1) = m^{\lambda-1}$. Since $X$ has $m$-osculating regularity we can degenerate $T$, in a family parametrized by a smooth curve, to a linear space $U_1$ contained in 
$$V_1 = \left\langle T^{3}_{p_1^1}, T^3_{p_1^{m+1}},\dots, T^3_{p_1^{m^{\lambda-1}-m+1}}\right\rangle$$ 
Again, since $X$ has $m$-osculating regularity we may specialize, in a family parametrized by a smooth curve, the linear space $V_1$ to a linear space $U_2$ contained in
$$V_2 = \left\langle T^{7}_{p_1^1}, T^7_{p_1^{m^2+1}},\dots, T^7_{p_1^{m^{\lambda-1}-m^2+1}}\right\rangle$$
Proceeding recursively in this way in last step we get a linear space $U_{\lambda-1}$ which is contained in 
$$V_{\lambda-1} = T^{2^{\lambda}-1}_{p_1^1}$$
Now, more generally, let us assume that 
$$k_1+1 = 2^{\lambda_1}+\dots + 2^{\lambda_a}+\varepsilon$$
with $\varepsilon\in\{0,1\}$, and $\lambda_1 > \lambda_2 > \dots > \lambda_a\geq 1$. Then
$$h_m(k_1) = m^{\lambda_1-1}+\dots + m^{\lambda_a-1}$$
By applying $a$ times the argument for $k_1+1 = 2^{\lambda}$ in the first part of the proof we may specialize $T$ to a linear space $U$ contained in 
$$V = \left\langle T^{2^{\lambda_1}-1}_{p_1^1}, T^{2^{\lambda_2}-1}_{p_1^{m^{\lambda_1-1}+1}},\dots, T^{2^{\lambda_a}-1}_{p_1^{m^{\lambda_1-1}+\dots+m^{\lambda_{a-1}-1}+1}}\right\rangle$$ 
Finally, using that $X$ has strong $2$-osculating regularity $a-1$ times we specialize $V$ to a linear space $U^{'}$ contained in 
$$V^{'} = T_{p_1^1}^{2^{\lambda_1}+\dots +2^{\lambda_a-1}}$$
Note that $T_{p_1^1}^{2^{\lambda_1}+\dots +2^{\lambda_a}-1} = T^{k_1}_{p_1^1}$ if $\varepsilon = 0$, and $T_{p_1^1}^{2^{\lambda_1}+\dots +2^{\lambda_a}-1} = T^{k_1-1}_{p_1^1}\subset T^{k_1}_{p_1^1}$ if $\varepsilon = 1$. In any case, since by hypothesis there is an $s$-dimensional linear subspace containing $\langle T_{p_1}^{k_1},\ldots,T_{p_l}^{k_l}\rangle$ that is not tangent to $X$ along a positive dimensional subvariety we conclude by Lemma \ref{lemma_sc}. 
\end{proof}

\section{On tangential weak defectiveness of Segre-Veronese varieties}\label{wd-SV}

Let $\textbf{\textit{n}}=(n_1,\dots,n_r)$ and $\textbf{\textit{d}} = (d_1,\dots,d_r)$ be two $r$-uples of positive integers, with $n_1\leq \dots \leq n_r$ and $d=d_1+\dots+d_r\geq 3$. 
Let $SV^{\textbf{\textit{n}}}_{\textbf{\textit{d}}}\subset\mathbb{P}^{N(\textbf{\textit{n}},\textbf{\textit{d}})}$, where $N(\textbf{\textit{n}},\textbf{\textit{d}}) = \prod_{i=1}^r\binom{n_i+d_i}{d_i}-1$, be the corresponding Segre-Veronese variety that is the product $\mathbb{P}^{n_1}\times\dots\times\mathbb{P}^{n_r}$ embedded by the complete 
linear system  $\big|\mathcal{O}_{\mathbb{P}^{n_1}\times\dots\times\mathbb{P}^{n_r}}(d_1,\dots, d_r)\big|$. We recall the notion of distance for Segre-Veronese varieties given in \cite[Definition 2.4]{AMR17}.

\begin{Definition}\label{distance}
Let $n$ and $d$ be positive integers, and set 
$$
\Lambda_{n,d}=\{I=\{i_1,\dots,i_{d}\},0\leq i_1 \leq \dots \leq i_{d} \leq n\}
$$
For $I,J\in \Lambda_{n,d}$, we define their distance $d(I,J)$ as the number of different coordinates.
More precisely, write $I=\{i_1,\dots,i_{d}\}$ and $J=\{j_1,\dots,j_{d}\}$. 
There are $r\geq 0$ distinct indexes $\lambda_1, \dots, \lambda_r\subset \{1, \dots, d\}$ and 
distinct indexes $\tau_1, \dots, \tau_r\subset \{1, \dots, d\}$ such that $i_{\lambda_k}=j_{\tau_k}$ for every $1\leq k\leq r$, and 
$$\{i_\lambda \: | \: \lambda\neq \lambda_1, \dots, \lambda_r\}\cap \{j_\tau \: | \: \tau\neq \tau_1, \dots, \tau_r\}=\emptyset$$
Then $d(I,J)=d-r$. Now, set 
$$
\Lambda=\Lambda_{\textbf{\textit{n}},\textbf{\textit{d}}}=\Lambda_{n_1,d_1}\times \dots \times \Lambda_{n_r,d_r}
$$
For $I=(I^1,\dots,I^r),J=(J^1,\dots,J^r)\in \Lambda$, 
we define their distance as 
$$
d(I,J)=d(I^1,J^1)+\dots+d(I^r,J^r)
$$
\end{Definition}
Such a distance, called the Hamming distance, was defined in \cite[Section 2]{CGG02} for Segre varieties. We will denote the homogeneous coordinates and the corresponding coordinate points of $\p^{N(\textbf{\textit{n}},\textbf{\textit{d}})}$ by $X_J$ and $e_J$ respectively, for $J\in\Lambda$. 

\begin{Proposition}\label{Prop1}
Let $p_0,\dots,p_{n_1}\in SV_{\textbf{\textit{d}}}^{\textbf{\textit{n}}}$ be general points. If $d:=\min\{d_1,\ldots,d_r\}\geq 2$ then a general hyperplane $H\subset\p^N$ containing $T=\langle T_{p_0}^{d-1}SV_{\textbf{\textit{d}}}^{\textbf{\textit{n}}},\ldots,T_{p_{n_1}}^{d-1}SV_{\textbf{\textit{d}}}^{\textbf{\textit{n}}}\rangle$ is not tangent to $SV_{\textbf{\textit{d}}}^{\textbf{\textit{n}}}$ along a positive dimensional subvariety.
\end{Proposition}
\begin{proof}
Since $PGL(n_{1}+1)\times\dots\times PGL(n_{r}+1)$ acts transitively on $SV_{\textbf{\textit{d}}}^{\textbf{\textit{n}}}$ we may assume that $p_i=e_{I_i}$, where $I_i=(\{i,\ldots,i\},\ldots,\{i,\ldots,i\})$. By \cite[Proposition 2.5]{AMR17} $T_{e_{I_i}}^{d-1}=\langle e_J\:|\:d(I_i,J)\leq d-1\rangle$, and hence
$$\begin{array}{ccl}
\langle T_{e_{I_0}}^{d-1},\ldots,T_{e_{I_{n_1}}}^{d-1}\rangle&=&\langle e_J\:|\:d(I_i,J)\leq d-1\:\:\text{ for some }\:i=0,\ldots n_1\rangle\\
&=&\{X_J=0\:|\:d(I_i,J)>d-1\:\text{ for all }\:i=0,\ldots n_1\}
\end{array}$$
Now, let $H\subset\p^{N(\textbf{\textit{n}},\textbf{\textit{d}})}$ be a general hyperplane containing $T$. We have that $H$ is given by an equation of type
\stepcounter{thm}
\begin{equation}\label{Eq.H}
\sum_{J\in \Lambda\:|\:d(I_i,J)>d-1, \forall\: i = 0,\dots,n_1}\alpha_J X_J=0,\:\:\alpha_J\in \C
\end{equation}
Let us denote by $\p^{N(\textbf{\textit{n}},\textbf{\textit{d}})-\dim(T)-1}$ the projective space whose homogeneous coordinates are the $\alpha_J$ with $J\in \Lambda$ and $d(I_i,J)>d-1$ for all $i=0,\ldots,n_1$. Now, for each fixed $i=0,\ldots,n_1$ we consider the following subset of $\Lambda$: for each $1\leq l\leq r$ and $0\leq j\leq n_l$ with $j\neq i$ let
$$
J_{i,j,l}=(J_1,\ldots,J_r)\in\Lambda\:\text{ where }\:J_l=\{j,\ldots,j\}\:\text{  and }\:J_{k}=\{i,\ldots,i\}\:\text{ for }\:k\neq l
$$
and set $\Lambda_i=\{J_{i,j,l}\in\Lambda\:|\:\text{ for all }1\leq l\leq r\:\text{ and } 0\leq j\leq n_l\:\text{ with }\:j\neq i\}\:$.

Observe that, since $d = \min\{d_i\}$ and $j\neq i$, each $J\in \Lambda_i$ satisfies $d(I_i,J) \geq d > d-1$ for all $i=0,\ldots,n_1$. Consider the projection
$$\begin{array}{ccccl}
\pi_i&:&\p^{N(\textbf{\textit{n}},\textbf{\textit{d}})-\dim(T)-1}&\dasharrow&\p^{\sum_{i\neq j}n_j}\\
&&(\alpha_J)_{J\in\Lambda\:|\:d(I_l,J)>d-1\:l=0,\ldots,n_1}&\longmapsto&(\alpha_J)_{J\in\Lambda_i}
\end{array}$$
the point $[1:\dots:1]\in\p^{\sum_{j\neq i}n_j}$ and let $H\in\pi_i^{-1}([1:\dots:1])$ be the hyperplane given by $\sum_{J\in \Lambda_i} X_J=0$. The intersection $H\cap SV_{\textbf{\textit{d}}}^{\textbf{\textit{n}}}$ corresponds to the hypersurface 
\stepcounter{thm}
\begin{equation}\label{gen_equ}
\sum_{J\in \Lambda_i} X_{1,i}^{d_1}\cdots X_{l,j}^{d_l}\cdots X_{r,i}^{d_r}=0
\end{equation}
where $X_{l,j}$ for $j=0,\ldots,n_l$ are the homogeneous coordinates on $\p^{n_l}$. Thus, in the affine chart $X_{1,i}=\cdots=X_{r,i}=1$ equation (\ref{gen_equ}) becomes
\stepcounter{thm}
\begin{equation}\label{local_equ}
\sum_{\substack{1\leq l\leq r\\ 0\leq j\leq n_l,\:j\neq i}} X_{l,j}^{d_l}=0
\end{equation}
The singular locus of $H\cap SV_{\textbf{\textit{d}}}^{\textbf{\textit{n}}}$ in the affine chart $X_{1,i}=\cdots=X_{r,i}=1$ is given by the following system of equations
$$\{d_lX_{l,j}^{d_l-1}=0\}_{1\leq l\leq r, \:0\leq j \leq n_l, \:j\neq i}$$
The only solution of this system is $X_{l,j}=0$, and so the hypersurface (\ref{local_equ}) is singular only at $p_{0}=(0,\dots,0)$. Therefore, we conclude that the intersection of $SV_{\textbf{\textit{d}}}^{\textbf{\textit{n}}}$ with a general hyperplane $H$ containing $T$ is singular, in a neighborhood of $p_0$, only at $p_{0}$. Since this argument holds for each $i=0,\ldots,n_1$ 
using Lemma \ref{lemma_sc} we get the claim.
\end{proof}

\begin{Proposition}\label{Prop2}
Let $p_0,\dots,p_{n_1}\in SV_{\textbf{\textit{d}}}^{\textbf{\textit{n}}}$ be general points and assume that $d=d_1\leq d_i-2$ for each $i\neq 1$. Then a general hyperplane $H\subset\p^N$ containing $T = \langle T_{p_0}^{d}SV_{\textbf{\textit{d}}}^{\textbf{\textit{n}}},\dots,T_{p_{n_1}}^{d}SV_{\textbf{\textit{d}}}^{\textbf{\textit{n}}}\rangle$ is not tangent to $SV_{\textbf{\textit{d}}}^{\textbf{\textit{n}}}$ along a positive dimensional subvariety.
\end{Proposition}
\begin{proof}
As in Proposition \ref{Prop1} we may assume that $p_i=e_{I_i}$, with $I_i=(\{i,\ldots,i\},\ldots,\{i,\ldots,i\})$. By \cite[Proposition 2.5]{AMR17} $T_{e_{I_i}}^{d}=\langle e_J\:|\:d(I_i,J)\leq d\rangle$. Hence
$$\begin{array}{ccl}
\langle T_{e_{I_0}}^{d},\ldots,T_{e_{I_{n_1}}}^{d}\rangle&=&\langle e_J\:|\:d(I_i,J)\leq d\:\:\text{ for some }\:i=0,\ldots n_1\rangle\\
&=&\{X_J=0\:|\:d(I_i,J)>d\:\text{ for all }\:i=0,\ldots n_1\}
\end{array}$$
Now, let $H\subset\p^{N(\textbf{\textit{n}},\textbf{\textit{d}})}$ be a general hyperplane containing $T$. We have that $H$ is given by an equation of type
$$
\sum_{J\in \Lambda\:|\:d(I_i,J)>d, \forall\: i = 0,\dots,n_1}\alpha_J X_J=0,\:\:\alpha_J\in \C
$$
Let us denote by $\p^{N(\textbf{\textit{n}},\textbf{\textit{d}})-\dim(T)-1}$ the projective space whose homogeneous coordinates are the $\alpha_J$ with $J\in \Lambda$ and $d(I_i,J)>d$ for all $i=0,\dots,n_1$. Now, for each fixed $i=0,\ldots,n_1$ we consider the following subset of $\Lambda$: for each $2\leq l\leq r$ and $0\leq j\leq n_l$ with $j\neq i$ set 
$$
J_{i,j,l}=(J_1,\ldots,J_r)\in\Lambda\:\text{ where }\:J_l=\{i,j,\ldots,j\},\: J_{k}=\{i,\ldots,i\}\:\text{ for }\:k\neq l
$$
and $\Lambda_{i,1}=\{J_{i,j,l}\in\Lambda\:|\:\text{ for all }\:j,l\neq i\}$.

Moreover, we also consider another subset of $\Lambda$ defined as follows: for each $0\leq j\leq n_1$ with $j\neq i$ let
$$J_{i,j}=(J_1,\ldots,J_r)\in\Lambda\text{ where }\:J_1=\{j,\ldots,j\},\:\:J_2=\{j,i,\ldots,i\},\:J_{k}=\{i,\ldots,i\}\:\text{ for }\:k\neq 1,2$$
and $\Lambda_{i,2}=\{J_{i,j,l}\in\Lambda\:|\:\text{ for all }\:j,l\neq i\}\:$, $\Lambda_i=\Lambda_{i,1}\cup\Lambda_{i,2}$.

Observe that, since $d = d_1 < d_i-2$ for $i \neq 1$ and $j\neq i$, each $J\in \Lambda_i$ satisfies $d(I_l,J)\geq d+1 > d$ for all $l=0,\ldots,n_1$. Therefore, we have a projection
$$\begin{array}{ccccl}
\pi_i&:&\p^{N(\textbf{\textit{n}},\textbf{\textit{d}})-\dim(T)-1}&\dasharrow&\p^{\sum_{i\neq j}n_j}\\
&&(\alpha_J)_{J\in\Lambda\:|\:d(I_l,J)>d\:l=0,\ldots,n_1}&\longmapsto&(\alpha_J)_{J\in\Lambda_i}
\end{array}$$
Now, consider the point $[1:\cdots:1]\in\p^{\sum_{j\neq i}n_j}$ and let $H\in\pi_i^{-1}([1:\cdots:1])$ be the hyperplane given by 
$$
\sum_{J\in \Lambda_i} X_J=0
$$
The intersection $H\cap SV_{\textbf{\textit{d}}}^{\textbf{\textit{n}}}$ corresponds to the hypersurface 
\stepcounter{thm}
\begin{equation}\label{gen_equ1}
\sum_{J\in \Lambda_{i,1}} X_{1,i}^{d_1}\cdots X_{l,i} X_{l,j}^{d_l-1}\cdots X_{r,i}^{d_r}+\sum_{J\in \Lambda_{i,2}} X_{1,j}^{d_1}X_{2,j}X_{2,i}^{d_2-1} X_{3,i}^{d_3}\cdots X_{r,i}^{d_r}=0
\end{equation}
where $X_{j,i}$, $i=0,\ldots,n_j$, are the homogeneous coordinates on $\p^{n_j}$. Thus, in the affine chart $X_{1,i}=\cdots=X_{r,i}=1$ the equation (\ref{gen_equ1}) becomes
\stepcounter{thm}
\begin{equation}\label{local_equ1}
F = \sum_{\substack{2\leq  l\leq r\\0\leq j\leq n_l,\:j\neq i}} X_{l,j}^{d_l-1}+\sum_{0\leq j\leq n_1,\:j\neq i} X_{1,j}^{d_1}X_{2,j}=0
\end{equation}
The system of the partial derivatives of $F$ is given by
$$\left\{\begin{array}{l}
d_{1}X_{1,j}^{d_{1}-1}X_{2,j}=0\\
(d_{2}-1)X_{2,j}^{d_{2}-2}+X_{1,j}^{d_{1}}=0\\
(d_l-1)X_{l,j}^{d_l-2}=0,\:l=3,\ldots,r\:\text{ and }\:j\neq i
\end{array}\right.$$
This system has a solution only when all the coordinates $X_{l,j}$ vanish, and so the hypersurface $\{F=0\}$ in (\ref{local_equ1}) is singular only at $p_{0}=(0,\dots,0)$. Therefore, we conclude that for a general hyperplane $H$ containing $T$ the hypersurface $H\cap SV_{\textbf{\textit{d}}}^{\textbf{\textit{n}}}$ is singular, in a neighborhood of $p_0$, only at $p_{0}$. Since this argument holds for each $i=0,\ldots,n_1$ using Lemma \ref{lemma_sc} we get the statement. 
\end{proof}

\begin{thm}\label{Bound_non_Wd}
Set $d := \min\{d_1,\ldots,d_r\}$. If
\begin{itemize}
\item[-] $h\leq (n_1+1)h_{n_1+1}(d-1)$ or
\item[-] $h\leq (n_1+1)h_{n_1+1}(d)$ and $d=d_1\leq d_i-2$ for each $2\leq i\leq r$
\end{itemize} 
then $SV_{\textbf{\textit{d}}}^{\textbf{\textit{n}}}$ is not $h$-weakly defective.
\end{thm}
\begin{proof}
Since by \cite[Propositions 5.1, 5.10]{AMR17} the Segre-Veronese variety $SV_{\textbf{\textit{d}}}^{\textbf{\textit{n}}}$ has strong $2$-osculating regularity and $(n_1+1)$-osculating regularity, the statement follows immediately from Propositions \ref{Prop1}, \ref{Prop2} and Theorem \ref{Non_Weakly_defec}.
\end{proof}

\begin{Remark}\label{asy_wd}
Write $d=2^{\lambda_1}+2^{\lambda_2}\ldots+2^{\lambda_s}+\epsilon$ with $\lambda_1>\lambda_2>\ldots>\lambda_s\geq 1$ and $\epsilon\in\{0,1\}$, so that $\lambda_1=\lfloor \log_2(d)\rfloor$. The first part of Theorem \ref{Bound_non_Wd} says that $SV_{\textbf{\textit{d}}}^{\textbf{\textit{n}}}$ is not $h$-weakly defective for $h\leq (n_1+1)((n_1+1)^{\lambda_1-1}+(n_1+1)^{\lambda_2-1}+\cdots+(n_1+1)^{\lambda_s-1})$. 

Now, write $d+1=2^{\lambda_1}+2^{\lambda_2}\ldots+2^{\lambda_s}+\epsilon$ with $\lambda_1>\lambda_2>\ldots>\lambda_s\geq 1$ and $\epsilon\in\{0,1\}$, hence $\lambda_1=\lfloor \log_2(d+1)\rfloor$. The second part of Theorem \ref{Bound_non_Wd} yields that $SV_{\textbf{\textit{d}}}^{\textbf{\textit{n}}}$ is not $h$-weakly defective for $h\leq (n_1+1)((n_1+1)^{\lambda_1-1}+(n_1+1)^{\lambda_2-1}+\cdots+(n_1+1)^{\lambda_s-1})$. Therefore, we have that asymptotically for 
$$h\leq (n_1+1)^{\lfloor \log_2(d)\rfloor}$$
$SV_{\textbf{\textit{d}}}^{\textbf{\textit{n}}}$ is not $h$-weakly defective. 
\end{Remark}

\subsection{On $1$-weak defectiveness of Segre-Veronese varieties}
In this section we give condition ensuring that Segre-Veronese varieties are not $1$-weakly defective. Note that this yields that their dual varieties are hypersurfaces.

\begin{Proposition}\label{Prop0}
If $n_r\leq \sum_{i=1}^{r-1}n_i$ then $SV_{\textbf{\textit{d}}}^{\textbf{\textit{n}}}$ is not $1$-weakly defective. 
\end{Proposition}
\begin{proof}
First of all, let us consider the Segre embedding of $\p^{n_1}\times\cdots\times\p^{n_r}$, that is $\textbf{\textit{d}} =(1,\dots,1)$. Let $p\in\p^{n_1}\times\cdots\times\p^{n_r}$ be a general point, without loss of generality we may assume that $p=e_{0,\ldots,0}$. Hence $T_p(\p^{n_1}\times\cdots\times\p^{n_r})=\langle e_J\:|\:d(J,(\{0\},\ldots,\{0\}))\leq 1\rangle$. Thus, a general hyperplane containing  $T_{p}(\p^{n_1}\times\cdots\times\p^{n_r})$ is given by an equation of type
$$\sum_{J\in\Lambda\:|\:d(J,(\{0\},\ldots,\{0\}))\geq 2}\alpha_{J}X_J=0$$
where $\Lambda$ is the set of indexes of the standard Segre variety. On the affine chart $X_{1,0}=\cdots=X_{r,0}=1$, where $X_{i,0},\ldots,X_{i,{n_i}}$ are homogeneous coordinates of $\p^{n_i}$, we have that $H\cap(\p^{n_1}\times\cdots\times\p^{n_r})$ is the hypersurface in $\C^{\sum n_i}$ given by
\stepcounter{thm}
\begin{equation}\label{Hypersurface}
\sum_{J=(\{j_1\},\ldots,\{j_r\})\in\Lambda\:|\:d(J,(\{0\},\ldots,\{0\}))\geq 2}\alpha_{J}X_{1,{j_1}}\cdots X_{r,{j_r}}=0
\end{equation}
where in the above formula whenever some of the variables $X_{1,0},\dots,X_{r,0}$ appear we set them equal to one. Note that for a general choice of the $\alpha_{J}$ the hypersurface defined by \ref{Hypersurface} has $0$-dimensional singular locus, since by \cite[Theorem 2.1]{Ott13} the Segre variety $\p^{n_1}\times\cdots\times\p^{n_r}$ is not $1$-weakly defective.

From now on $\Lambda$ will be the set of indexes of a Segre-Veronese variety. Let $p\in SV_{\textbf{\textit{d}}}^{\textbf{\textit{n}}}$. As before without loss of generality we can assume that $p=e_{I_0}$. By \cite[Proposition 2.5]{AMR17} $T_{p}SV_{\textbf{\textit{d}}}^{\textbf{\textit{n}}}=\langle e_J\:|\:d(I_0,J)\leq 1\rangle$.
Observe that for each $J=(\{j_1\},\ldots,\{j_r\})$ such that $d(J,(\{0\},\ldots,\{0\}))\geq 2$ we can consider $J'=(J_1,\ldots,J_r)\in\Lambda$ where $J_i=\{0,\ldots,0,j_i\}$. Therefore, considering the hyperplane $H$ given by 
$$\sum_{J'}\alpha_{J}X_{J'}=0$$
where we set $X_{1,0}=\cdots=X_{r,0}=1$ whenever these variables appear in the expression above, we see that in the affine chart $X_{1,0}=\cdots=X_{r,0}=1$ the hypersurface $H\cap SV_{\textbf{\textit{d}}}^{\textbf{\textit{n}}}$ in $\C^{\sum n_i}$ is given by (\ref{Hypersurface}). Thus, the statement follows from the first part of the proof.
\end{proof}

\begin{Proposition}\label{Prop1bis}
Assume that $n_r> \sum_{i=1}^{r-1}n_i$.
\begin{itemize}
\item[-] If $d_r\geq 2$ then $SV_{\textbf{\textit{d}}}^{\textbf{\textit{n}}}$ is not $(n_1+1)$-weakly defective.
\item[-] If $d_r=1$ then $SV_{\textbf{\textit{d}}}^{\textbf{\textit{n}}}$ is $1$-weakly defective.
\end{itemize} 
\end{Proposition}
\begin{proof}
Let $p_0,\ldots,p_{n_1}\in SV_{\textbf{\textit{d}}}^{\textbf{\textit{n}}}$ be general points. Without loss of generality, we can suppose that $p_i=e_{I_i}$. By \cite[Proposition 2.5]{AMR17} $T_{e_{I_i}}SV_{\textbf{\textit{d}}}^{\textbf{\textit{n}}}=\langle e_J\:|\:d(I_i,J)\leq 1\rangle$, and hence
$$\begin{array}{ccl}
T=\langle T_{e_{I_0}}^{1},\ldots,T_{e_{I_{n_1}}}^{1}\rangle&=&\langle e_J\:|\:d(I_i,J)\leq 1\:\:\text{ for some }\:i=0,\ldots n_1\rangle\\
&=&\{X_J=0\:|\:d(I_i,J)>1\:\text{ for all }\:i=0,\ldots n_1\}
\end{array}$$

Now, let $H\subset\p^{N(\textbf{\textit{n}},\textbf{\textit{d}})}$ be a general hyperplane containing $\langle T_{p_0}^{1},\ldots,T_{p_{n_1}}^{1}\rangle$. Then $H$ is given by  an equation of type
$$
\sum_{J\in \Lambda\:|\:d(I_i,J)>1, \forall\: i = 0,\dots,n_1}\alpha_J X_J=0,\:\:\alpha_J\in \C
$$
Let us denote by $\p^{N(\textbf{\textit{n}},\textbf{\textit{d}})-\dim(T)-1}$ the projective space whose homogeneous coordinates are the $\alpha_J$ with $J\in \Lambda$ and $d(I_i,J)>d$ for all $i=0,\dots,n_1$.

To prove the first claim let us fix $l\in\{0,\dots,n_1\}$. We will discuss in detail the case $l = 0$, the argument for the remaining values of $l$ is analogous. 

Let us consider the subset $\Lambda'\subset\Lambda$ given by the set of indexes $J'=(J_1,\ldots,J_r)$ where for each pair $i,j$ with $i\in\{1,\ldots,r-1\}$ and $1\leq j\leq n_i$ we set
$$J_i=\{0,\ldots,0,j\},\:J_r=\left\lbrace0,\ldots,0,1+j+\sum_{l<i}n_l\right\rbrace\:\text{ and }J_k=\{0,\ldots,0\} \:\text{ for }\:k\neq i,r$$
Furthermore, consider the subset $\Lambda''\subset\Lambda$ given by the set of indexes $J''=J_j=(J_1,\ldots,J_r)$ such that 
$$J_r=\{j,\ldots,j\},\:\text{ and }J_k=\{0,\ldots,0\} \:\text{ for }\:k\neq r$$
for each $2+\sum_{l\leq r-1}n_l\leq j\leq n_r$ and $j=1$.

Since $1\leq j<1+j+\sum_{l<i}n_l$, each $J\in\Lambda_0=\Lambda'\cup\Lambda''$ satisfies $d(I_i,J)>1$ for all $i=0,\ldots,n_1$. Thus, we have a natural projection
$$\begin{array}{ccccl}
\pi_l&:&\p^{N(\textbf{\textit{n}},\textbf{\textit{d}})-\dim(T)-1}&\dasharrow&\p^{n_r}\\
&&(\alpha_J)_{J\in\Lambda\:|\:d(I_i,J)>1\:i=0,\ldots,n_1}&\longmapsto&(\alpha_J)_{J\in\Lambda_0}
\end{array}$$

Now, consider the point $[1:\cdots:1]\in\p^{n_r}$ and let $H\in\pi_l^{-1}([1:\cdots:1])$ be the hyperplane given by
$$
\sum_{J\in\Lambda_0}X_J=0
$$
In the affine chart $X_{1,0}=\cdots=X_{r,0}=1$, where for each $i\in\{1,\ldots,r\}$, $X_{i,0},\ldots,X_{i,{n_i}}$ are the homogeneous coordinates on $\p^{n_i}$, we have that $H\cap SV_{\textbf{\textit{d}}}^{\textbf{\textit{n}}}$ is the hypersurface in $\C^{\sum n_i}$ given by
$$\sum_{\substack{1\leq i\leq r-1\\1\leq j\leq n_i}}X_{i,j}X_{r,{j+1+\sum_{l<i}n_l}}+\sum_{2+\sum_{l\leq r-1}n_l\leq j\leq n_r}X_{r,j}^{d_r}+X_{r,1}^{d_r}=0$$
Looking at the system of the partial derivatives we see that this hypersurface is singular only at $(0,\dots,0)$. Therefore, using Lemma \ref{lemma_sc} we prove the first claim.
For the second part, let us consider a general hyperplane $H$ that contains $T_{e_{I_0}}SV_{\textbf{\textit{d}}}^{\textbf{\textit{n}}}$. Hence, $H$ is the zero locus of a polynomial $F$ of the form
$$
F=\sum_{J\in\Lambda\:|\:d(J,I_0)\geq 2}\alpha_JX_J,\:\:\alpha_J\in \C
$$
In the affine chart $X_{1,0}=\cdots=X_{r,0}=1$ the intersection $H\cap SV_{\textbf{\textit{d}}}^{\textbf{\textit{n}}}$ is the hypersurface in $\C^{\sum n_i}$ given by
$$\widetilde{F}=\sum_{J=(J_1,\ldots,J_{r-1},\{j\})\in\Lambda\:|\:d(J,I_0)\geq 2}\alpha_J X_{1,{J_1}}\cdots X_{r,j}=0$$
where with $X_{1,{J_k}}$ we denote the product of powers of the homogeneous coordinates on $\mathbb{P}^{n_l}$ with exponents given by the $J_k$. Observe that for each $1\leq i\leq r-1$ and $1\leq j\leq n_i$ we have $$\frac{\partial \widetilde{F}}{\partial X_{i,j}}=(\sum_{k=1}^{n_r}\alpha_{i,j}^{k}X_{r,k}+G_{k}(X_{1,1},\ldots,X_{r-1,n_{r-1}})X_{r,k})+G(X_{1,1},\ldots,X_{r-1,n_{r-1}})$$ and for each $1\leq k\leq n_r$ we have $$\frac{\partial \widetilde{F}}{\partial X_{r,k}}=G'(X_{1,1},\ldots,X_{r-1,n_{r-1}})$$ with $G_{k}(X_{1,1},\ldots,X_{r-1,n_{r-1}})$, $G(X_{1,1},\ldots,X_{r-1,n_{r-1}})$ and $G'(X_{1,1},\ldots,X_{r-1,n_{r-1}})$ polynomials with no constant terms since by assumption $d_r=1$.

Now, note that the locus given by $X_{1,1}=X_{1,2}=\cdots=X_{r-1,n_{r-1}-1}=X_{r-1,n_{r-1}}=0$ and
$$\sum_{k=1}^{n_r}\alpha_{1,1}^{k}X_{r,k}=\sum_{k=1}^{n_r}\alpha_{1,2}^{k}X_{r,k} = \dots = \sum_{k=1}^{n_r}\alpha_{r-1,r-1}^{k}X_{r,k}=0$$
is contained in the singular locus of $\{\widetilde{F}=0\}$. Therefore, we get a linear system in $n_r$ variables and $\sum_{i=1}^{r-1}n_i$ equations. Since $n_r>\sum_{i=1}^{r-1}n_i$ we conclude that the singular locus of $H\cap SV_{\textbf{\textit{d}}}^{\textbf{\textit{n}}}$ contains at least a linear space of dimension $n_r-\sum_{i=1}^{r-1}n_i>0$ yielding that $SV_{\textbf{\textit{d}}}^{\textbf{\textit{n}}}$ is $1$-weakly defective.
\end{proof}

By Proposition \ref{Prop1bis} we have that $SV_{\textbf{\textit{d}}}^{\textbf{\textit{n}}}$ with $\textbf{\textit{n}}=(1,n)$ and $\textbf{\textit{d}}=(d,1)$ is $1$-weakly defective. Now, we determine the smallest dimension of a linear subspace tangent to $SV_{\textbf{\textit{d}}}^{\textbf{\textit{n}}}$ along a positive dimensional subvariety. 

\begin{Proposition}\label{Prop3bis}
Let $SV_{\textbf{\textit{d}}}^{\textbf{\textit{n}}}$ with $\textbf{\textit{n}}=(1,n)$ and $\textbf{\textit{d}}=(d,1)$. Then $SV_{\textbf{\textit{d}}}^{\textbf{\textit{n}}}$ is not $(1,s)$-tangentially weakly defective if and only if $s\leq d(n+1)$.
\end{Proposition}
\begin{proof} 
Let $p\in SV_{\textbf{\textit{d}}}^{\textbf{\textit{n}}}$ be a general point, without loss the generality we can suppose that $p=e_{\{0,\ldots,0\},\{0\}}$. Then we have $T_pSV_{\textbf{\textit{d}}}^{\textbf{\textit{n}}}=\langle e_J\:|\:d(J,(\{0,\ldots,0\},\{0\}))\leq 1\rangle$.

Now, let $\Pi\subset \p^{dn+d+n}$ be a general linear subspace of dimension $s$ such that $T_pSV_{\textbf{\textit{d}}}^{\textbf{\textit{n}}}\subset \Pi$. Therefore, we may write $\Pi =\bigcap_{i=1,\dots,dn+d+n-s}H_{i}$, where the $H_{i}$ are general hyperplanes tangent to $SV_{\textbf{\textit{d}}}^{\textbf{\textit{n}}}$ at $p$. We have that $\Pi\cap SV_{\textbf{\textit{d}}}^{\textbf{\textit{n}}}$ is given by
$$
\begin{small}
\left\{\begin{array}{l}
F_{1}=\sum_{\substack{1\leq  i\leq d\\1\leq j\leq n}}\alpha_{i,j}^{1}X_{0}^{d-i}X_{1}^{i}Y_{j}+\sum_{2\leq i\leq d}\alpha_{i,0}^{1}X_{0}^{d-i}X_{1}^{i}Y_{0}=0\\
\vdots \\
F_{dn+d+n-s} = \sum_{\substack{1\leq  i\leq d\\1\leq j\leq n}}\alpha_{i,j}^{dn+d+n-s}X_{0}^{d-i}X_{1}^{i}Y_{j}+\sum_{2\leq i\leq d}\alpha_{i,0}^{dn+d+n-s}X_{0}^{d-i}X_{1}^{i}Y_{0}=0
\end{array}\right.
\end{small}
$$
and working on the affine chart $X_{0}=Y_{0}=1$ we reduce to
\stepcounter{thm}
\begin{equation}\label{syschart}
\left\{\begin{array}{l}
F_{1}=\sum_{\substack{1\leq  i\leq d\\1\leq j\leq n}}\alpha_{i,j}^{1}X_{1}^{i}Y_{j}+\sum_{2\leq i\leq d}\alpha_{i,0}^{1}X_{1}^{i}=0\\
\vdots \\
F_{dn+d+n-s} = \sum_{\substack{1\leq  i\leq d\\1\leq j\leq n}}\alpha_{i,j}^{dn+d+n-s}X_{1}^{i}Y_{j}+\sum_{2\leq i\leq d}\alpha_{i,0}^{dn+d+n-s}X_{1}^{i}=0
\end{array}\right.
\end{equation}
Then, $\Sing(H_{1}\cap \dots \cap H_{dn+d+n-s}\cap SV_{\textbf{\textit{d}}}^{\textbf{\textit{n}}})$ contains the variety cut out by the following equations
\stepcounter{thm}
\begin{equation}\label{singloc}
\left\{\begin{array}{l}
\sum_{\substack{1\leq j\leq n}}\alpha_{1,j}^{1}Y_{j}=0\\
\vdots \\
\sum_{\substack{1\leq j\leq n}}\alpha_{1,j}^{dn+d+n-s}Y_{j}=0\\
X_{1}=0
\end{array}\right.
\end{equation}
and, for a general choice of the $\alpha_{i,j}^{k}$ we have that this is a linear space in the hyperplane $X_{1}=0$ of dimension $s-d(n+1)$.  

Now, consider a special linear space $\Pi$ such that (\ref{syschart}) takes the following form
$$
\left\{\begin{array}{l}
F_{1}=\sum_{\substack{1\leq j\leq n}}\alpha_{1,j}^{1}X_{1}Y_{j} = 0\\
\vdots \\
F_{dn+d+n-s} = \sum_{\substack{1\leq j\leq n}}\alpha_{1,j}^{dn+d+n-s}X_{1}Y_{j}= 0
\end{array}\right.
$$
Then $\{F_{1} = \dots = F_{dn+d+n-s} = 0\}$ splits as 
$$\{X_1=0\}\cup \{\sum_{\substack{1\leq j\leq n}}\alpha_{1,j}^{1}Y_{j} = \dots = \sum_{\substack{1\leq j\leq n}}\alpha_{1,j}^{dn+d+n-s}Y_{j} = 0\}$$ 
and its singular locus is exactly given by (\ref{singloc}). Now, Lemma \ref{lemma_sc} yields that a general linear space of dimension $s$ containing $T_pSV_{\textbf{\textit{d}}}^{\textbf{\textit{n}}}$ has contact locus of dimension at most $s-d(n+1)$. Hence, $SV_{\textbf{\textit{d}}}^{\textbf{\textit{n}}}$ is not $(1,s)$-tangentially weakly defective for $s\leq d(n+1)$.
\end{proof}

Following the line of proof of Proposition \ref{Prop3bis} we can prove the following result on $(1,s)$-tangential weak defectiveness.

\begin{Proposition}\label{Prop4bis}
Consider $SV_{\textbf{\textit{d}}}^{\textbf{\textit{n}}}$ with $\textbf{\textit{n}}=(n_{1},\dots,n_{r})$ and $\textbf{\textit{d}}=(d_{1},\dots,d_{r-1},1)$, and assume that $n_{r}> \sum_{i=1}^{r-1} n_{i}$. If 
$$s\leq \prod_{i=2}^r{\binom{n_i+d_i}{n_i}}-n_{r}\sum_{i=1}^{r-1}n_{i}$$ 
then $SV_{\textbf{\textit{d}}}^{\textbf{\textit{n}}}$ is not $(1,s)$-tangentially weakly defective.
\end{Proposition} 
\begin{proof}
Without loss of generality we can assume as usual that $p=e_{J_{0}} \in SV_{\textbf{\textit{d}}}^{\textbf{\textit{n}}}$ where $J_{0}=(\{0,\dots,0\},\dots,\{0,\dots,0\})$. A basis for the linear system of the hyperplanes containing $T_{p}SV_{\textbf{\textit{d}}}^{\textbf{\textit{n}}}$ is given by $$\{X_{1,{J_1}}\dots X_{r-1,J_{r-1}}X_{r,j}=0\}_{{J=\{J_1,\ldots,J_{r-1},\{j\}\}\in\Lambda\:|\:d(J,I_0)\geq 2}}$$
Now let us consider hyperplane sections of the form $$F_{i,j,l}=X_{1,0}^{d_{1}}\dots X_{i,j}X_{i,0}^{d_{i}-1}\dots X_{r,l}=0$$ for $1\leq i \leq r-1$,$1 \leq j \leq n_{i}$ and $1 \leq l \leq n_{r}$.

In the affine chart $\mathbb{C}^{\sum_{i=1}^{r}n_{i}}$ defined by $X_{1,0}=\dots=X_{r,0}=1$ the partial derivatives of $F_{i,j,l}$ are given by
$$\frac{\partial(X_{1,0}^{d_{1}}\dots X_{i,j}X_{i,0}^{d_{i}-1}\dots X_{r,l})}{\partial X_{i,j}}=X_{r,l},\:\frac{\partial{(X_{1,0}^{d_{1}}\dots X_{i,j}X_{i,0}^{d_{i}-1}\dots X_{r,l})}}{\partial X_{r,l}}=X_{i,j}$$ 
Then the Jacobian matrix of the $F_{i,j,l}$ has rank zero if and only if all the coordinates $X_{i,j}$ with $1 \leq j \leq n_{i}$ vanish.
In particular, the intersection of the special hyperplane sections 
$$X_{1,0}^{d_{1}}\dots X_{i,j}X_{i,0}^{d_{i}-1}\dots X_{r,l}=0$$
has a singularity spanning the whole of $\mathbb{C}^{\sum_{i=1}^{r}n_{i}}$ only at $(0,\dots,0)$. Now, to conclude it is enough to note that the number of these special hyperplane sections is $n_{r}\sum_{i=1}^{r-1}n_{i}$ and to apply Lemma \ref{lemma_sc}.
\end{proof}

Finally, we have the following classification of $1$-weakly defective Segre-Veronese varieties.

\begin{thm}\label{th1wd}
The Segre-Veronese $SV_{\textbf{\textit{d}}}^{\textbf{\textit{n}}}$ is $1$-weakly defective if and only if $d_r=1$ and $n_r>\sum_{i=1}^{r-1}n_i$.
\end{thm}
\begin{proof}
It is an immediate consequence of Propositions \ref{Prop0}, \ref{Prop1bis}.
\end{proof}

\section{On tangential weak defectiveness of products}\label{wdfib}
In this section we study tangential weak defectiveness for varieties that can be written as a product of a smaller dimensional variety and the projective line.

\begin{Lemma}\label{comp}
Let $W\subseteq\mathbb{P}^m$ be a non-degenerated irreducible projective variety, and consider the Segre embedding of $X = W\times \mathbb{P}^r\subseteq \mathbb{P}^m\times\mathbb{P}^r\rightarrow\mathbb{P}^N$ with $N = rm+r+m$. Fix a point $p\in\mathbb{P}^r$ and a hyperplane $H\subset\mathbb{P}^r$ not passing through $p$. Let $Z = W\times\{p\}$, $Y = W\times H$, and denote by $H_Z = \left\langle Z\right\rangle$, $H_Y = \left\langle Y\right\rangle$ their linear spans. Then $H_Z$ and $H_Y$ are complementary subspaces of $\mathbb{P}^N$, and $X\cap H_Z = Z$, $X\cap H_Y = Y$. 
\end{Lemma}
\begin{proof}
Since $W\subseteq\mathbb{P}^m$ is non-degenerated we have that $H_Z = \left\langle\mathbb{P}^m\times \{p\}\right\rangle$ and $H_Y = \left\langle\mathbb{P}^m\times H\right\rangle$. Consider homogeneous coordinates $[x_0:\dots:x_r]$ on $\mathbb{P}^r$ and $[y_0:\dots:y_m]$ on $\mathbb{P}^m$. Without loss of generality we may assume that $p = [1:0:\dots :0]$ and $H = \{x_0=0\}$. Hence, $H_Z = \{z_{0,1}=\dots = z_{m,r}=0\}$ and $H_Y = \{z_{0,0}=\dots = z_{m,0}=0\}$, where $z_{i,j}$ is the homogeneous coordinate on $\mathbb{P}^N$ corresponding to $y_ix_j$. Hence $H_Z$ and $H_Y$ are complementary subspaces of $\mathbb{P}^N$. 

Now, assume that there is a point $q\in X\cap H_Z$ with $q\notin Z$. Since $X = W\times\mathbb{P}^r$ the point $q$ lies on a fiber $\mathbb{P}^{r}_{w}$ over a point $w\in W$. Such fiber intersects $Z$ in a points $z\in Z$ with $z\neq q$ and hence $\mathbb{P}^r_w$ intersects $H_Z$ in at least two distinct points. On the other hand, note that $H_Z = \left\langle\mathbb{P}^m\times \{p\}\right\rangle$ is the fiber $\mathbb{P}^m_p$ over $p$ of the projection $\mathbb{P}^m\times\mathbb{P}^r\rightarrow\mathbb{P}^r$. A contradiction. 

Similarly, assume that there is a point $q\in X\cap H_Y$ with $q\notin Y$. The point $q$ lies on a fiber $\mathbb{P}^r_{w}$ over a point $w\in W$. Hence $\mathbb{P}^r_{w}$ intersects $Y$ in a hyperplane $H_w$ of $\mathbb{P}^r_{w}$ not containing $q$, and $H_Y$ contains the fiber $\mathbb{P}^r_{w} = \left\langle q, H_w\right\rangle$. A contradiction. 
\end{proof}

\begin{Proposition}\label{tang}
Let $W\subseteq\mathbb{P}^m$ be a non-degenerated irreducible projective variety, and consider the Segre embedding of $X = W\times\mathbb{P}^r\subseteq \mathbb{P}^m\times\mathbb{P}^r\rightarrow\mathbb{P}^N$ with $N = rm+r+m$.

If $p,q\in X$ are two distinct points lying on the same fiber of $\pi:X\rightarrow W$ over a smooth point $w\in W$ then the span of the tangent spaces $\left\langle T_pX,T_qX\right\rangle$ is tangent to $X$ along the line $\left\langle p, q\right\rangle$. 
\end{Proposition}
\begin{proof}
Let $w\in W$ be a smooth point. We can parametrize $W$ in a neighborhood of $W$ as
$$
\begin{array}{ccccl}
\varphi&:&\mathbb{
C}^d&\longrightarrow&\mathbb{C}^m\\
&&(x_1,\dots,x_d)&\longmapsto&(\phi_1(x_1,\dots,x_d),\dots,\phi_m(x_1,\dots,x_d))
\end{array}
$$ 
where $d = \dim(W)$ and $\phi(0) = w$. Hence, a parametrization of $X$ is given by
$$
\begin{array}{ccccl}
\psi&:&\mathbb{
C}^d\times\mathbb{C}^r&\longrightarrow&\mathbb{C}^N\\
&&((x_1,\dots,x_d),(1,y_1,\dots,y_r))&\longmapsto&(\phi_1,\dots,\phi_m,\phi_1y_1,\dots,\phi_my_r)
\end{array}
$$  
Let us set $a_{i,j} = \frac{\partial\phi_i}{\partial x_j}(0)$ and $b_k = \phi_k(0)$. Without loss of generality we may assume that $p = \psi((0,\dots,0),(1,0,\dots,0))$ and $p = \psi((0,\dots,0),(1,\dots,1))$ so that the line $\left\langle p,q\right\rangle$ is parametrized by $\gamma(t)=\psi((0,\dots,0),(1,t,\dots,t))$. Now, the tangent space of $X$ at $\gamma(t)$ is spanned by the rows of the following matrix
$$
\begin{small}
A(t) = \left(\begin{array}{ccccccccccccccc}
a_{1,1}t & \dots & a_{1,1}t & a_{2,1}t & \dots & a_{2,1}t & \dots & \dots & a_{m,1}t & \dots & a_{m,1}t & a_{1,1} & \dots & a_{m,1}\\ 
\vdots & \ddots & \vdots & \vdots & \ddots & \vdots & \ddots & \ddots & \vdots & \ddots & \vdots & \vdots & \ddots & \vdots\\ 
a_{1,d}t & \dots & a_{1,d}t & a_{2,d}t & \dots & a_{2,d}t & \dots & \dots & a_{m,d}t & \dots & a_{m,d}t & a_{1,d} & \dots & a_{m,d}\\
b_1 & \dots & 0 & b_2 & \dots & 0 & \dots & \dots & b_m & \dots & 0 & 0 & \dots & 0\\ 
\vdots & \ddots & \vdots & \vdots & \ddots & \vdots & \ddots & \ddots & \vdots & \ddots & \vdots & \vdots & \ddots & \vdots\\ 
0 & \dots & b_1 & 0 & \dots & b_2 & \dots & \dots & 0 & \dots & b_m & 0 & \dots & 0
\end{array}\right)
\end{small}
$$
and to conclude it is enough to observe that $A(t) = tA(1)-(t-1)A(0)$.
\end{proof}

Now, we are ready to prove our main result on tangential weak defectiveness of products. 

\begin{thm}\label{gen-twd}
Let $W\subseteq\mathbb{P}^m$ be a non-degenerated irreducible projective variety, and consider the Segre embedding of $X = W\times \mathbb{P}^1\subseteq \mathbb{P}^m\times\mathbb{P}^1\rightarrow\mathbb{P}^N$ with $N = 2m+1$. Assume that $W$ has $s-$osculating regularity and $2-$strong osculating regularity. 

If the following conditions are satisfied:
\begin{itemize}
\item[-] for a general point $w \in W$ the intersection $T_w^dW \cap W=S$ is a zero dimensional scheme supported on $w$;
\item[-] for a general choice of two points $p,q \in \mathbb{P}^1$ and a general hyperplane $H$ in $\left \langle W \times \{q\} \right \rangle$ containing $T_w^d (W \times \{q\})$ we have that $$\left \langle H, T_{\pi(w)}^{\frac{d-1}{2}}(W \times \{p\}) \right \rangle \cap W \times \{p\}=S$$ and $S$ is supported on the projection of $w$;
\item[-] $h_s(d)\dim(X)+h_s(d)-1 < m$;
\end{itemize}
then $X$ is not $(h_s(d),m+h_s(d)-1)$-tangentially weakly defective, and hence $X$ is $h_s(d)$-identifiable. In particular, under this bound $X$ is not $h_s(d)$-defective. 

\end{thm}
\begin{proof}
Take two distinct points $p,q\in\mathbb{P}^1$. Let $Z = W\times\{p\}$, $Y = W\times\{q\}$, $H_Z = \left\langle Z\right\rangle$, $H_Y = \left\langle Y\right\rangle$. Note that by Lemma \ref{comp} we have that $H_Z\cap Z_Y = \emptyset$, $\left\langle H_Z,H_Y\right\rangle = \mathbb{P}^N$, $X\cap H_Y = Y$, $X\cap H_Z = Z$.
Let $h:=h_s(d)$.
Fix $y_1,\dots,y_h\in Y$ general points, and let $z_1,\dots, z_h\in Z$ be their projections through the projection map $\pi:X\rightarrow Z$. Now, consider general points $x_1(t),\dots,x_h(t)\in X$ with $t\in\mathbb{C}^{*}$ such that $\lim_{t\mapsto 0} x_i(t) = y_i$, and let 
$$T_t = \left\langle T_{x_1(t)}X,\dots,T_{x_h(t)}X\right\rangle$$
Note that if $z_i(t) = \pi(x_i(t))$ then $\lim_{t\mapsto 0} z_i(t) = z_i$. Set $T_0 = \lim_{t\mapsto 0}T_t$. Since $\dim(T_0)\leq h\dim(X)+h-1$ and by hypothesis $h\dim(X)+h-1< m$ there exists a hyperplane $H_0\subset H_Y$ containing $T_0\cap H_Y$.

Let $\{H_t\}_{t\in\mathbb{C}^{*}}$ be a family of hyperplanes in $\mathbb{P}^m$ such that $H_t\supseteq T_t\cap\mathbb{P}^m$. Hence we have $T_t\subseteq\left\langle H_t,z_1(t),\dots,z_h(t)\right\rangle$, and since $H_0$ and $\left\langle z_1,\dots,z_h\right\rangle$ are disjoint and $z_1,\dots,z_h\in Z$ are general we have that 
$$\lim_{t\mapsto 0}\left\langle H_t,z_1(t),\dots,z_h(t)\right\rangle = \left\langle H_0,z_1,\dots,z_h\right\rangle$$ with $T_0\subseteq\left\langle H_0,z_1,\dots,z_h\right\rangle$.

Let $y_0$ be a general point of $Y$ and let $$\gamma_1,\dots,\gamma_h:C \rightarrow Y$$ be smooth curves such that $\gamma_j(t_0)=y_0$ and $\gamma_j(t_{\infty})=y_h$. By the hypotheses on osculating regularity we have that $\lim_{t \rightarrow 0}\left\langle T_{\gamma_1(t)}Y,\dots,T_{\gamma_h(t)} \right \rangle \subset T_{y_0}^dY$.
Furthermore the curves $\pi \circ \gamma_1, \dots, \pi \circ\gamma_h:C \rightarrow Z$ realizes the degeneration $$\lim_{t \rightarrow 0}\left\langle \pi(\gamma_1(t)),\dots,\pi(\gamma_h(t)) \right \rangle \subset T_{\pi(y_0)}^{\frac{d-1}{2}}Z$$

Thanks to Lemma \ref{lemma_sc} we have that $\left\langle H_0,z_1,\dots,z_h\right\rangle\cap H_Z = \left\langle z_1,\dots,z_h\right\rangle$ scheme theoretically in a neighbourhood of the $z_i$.

Assume that $\left\langle H_0,z_1,\dots,z_h\right\rangle$ is tangent to $X$ at a points $x\neq y_i$ for all $i = 1,\dots,h$. Then $\left\langle H_0,z_1,\dots,z_h\right\rangle$ contains all the fiber $\mathbb{P}^1_x = \pi^{-1}(x)$ and therefore the point $\mathbb{P}^1_{x}\cap Z$ which must then be one of the $z_i$, say $z_h$. Hence $x\in\mathbb{P}^1_{z_h}$. 

Now, Proposition \ref{tang} yields that $\left\langle H_0,z_1,\dots,z_h\right\rangle$ is tangent to $X$ along the line $\left\langle x,y_h\right\rangle = \mathbb{P}^1_{z_h}$, and in particular is tangent to $X$ at $z_h$, a contradiction. Therefore, $\left\langle H_0,z_1,\dots,z_h\right\rangle$ and hence $\left\langle H_t,z_1(t),\dots,z_h(t)\right\rangle$ and $T_t$ are tangent to $X$ just at the prescribed points $x_i(t)$ for $i = 1,\dots,h$.
\end{proof}

\begin{Remark}\label{sd-nn}
Note that the non secant defectiveness of $X$ is not needed anywhere in the proof of Theorem \ref{gen-twd}.
\end{Remark}

As an application to Segre-Veronese varieties we get the following result. 

\begin{Corollary}\label{CorSV}
Consider a Segre-Veronese variety $SV_{\textbf{\textit{d}}}^{\textbf{\textit{n}}}\subset\p^{N(\textbf{\textit{n}},\textbf{\textit{d}})}$ with $\textbf{\textit{n}}=(1,n_2,\dots,n_r)$ and $\textbf{\textit{d}}=(1,d_2,\dots,d_r)$.
Assume that $n_2 \leq n_3 \leq \dots \leq n_r$ and let $d:=\min\{d_i\}-1$. If $$h < h_{n_2+1}(d) \sim n_2^{\lfloor log_2(d) \rfloor}$$  
then $SV_{\textbf{\textit{d}}}^{\textbf{\textit{n}}}$ is not $h$-tangentially weakly defective, and hence $SV_{\textbf{\textit{d}}}^{\textbf{\textit{n}}}$ is $h$-identifiable. In particular, under this bound $SV_{\textbf{\textit{d}}}^{\textbf{\textit{n}}}$ is not $h$-defective. 
\end{Corollary}
\begin{proof}
Since $T^d_p (SV_{d}^{n}) \subset T^{d'}_p (SV_{d}^{n})$ for $d' \geq d$ and $T_p^{d_i}(V_{d_i}^{n_i}) \subset T_p^d (SV_{d}^{n})$ for every $i=1,\dots,r$ we can look only at the Veronese factor $V_{d_j}^{n_j}$ for which $d=d_j-1$.
In this case if $p=[x_0^d]$ for a suitable choice of coordinates $[x_0,\dots,x_{n_j}]$ in $\mathbb{P}^{n_j}$ we have that 
$$T_p^d V_{d_j}^{n_j}=\left \langle x_0F\:|\:\deg(F)=d-1 \right \rangle$$
and so $T_p^d V_{d_j}^{n_j} \cap V_{d_j}^{n_j}$ is supported on $p$.

We can assume that $p,q \in \p^1$ are given by $p=[0,1]$ and $q=[1,0]$.
For every $i=2, \dots,r$ let $I_i$ be the set of multi-indexes of size $|I_i|=d_i$ associated to the coordinates $[x_0^i,\dots,x_{n_i}^i]$ under the Veronese embedding given by $|\mathcal{O}_{\p^{n_i}}(d_i)|$. Finally let $[Z_{I_2,\dots,I_r,0},Z_{I_2,\dots,I_r,1}]_{(I_2,\dots,I_r)}$ be the coordinates of the Segre-Veronese embedding in $\p^{N(\textbf{\textit{n}},\textbf{\textit{d}})}$.
If $w=[\otimes_{j=2,\dots,r}(x_0^j)^{d_j}]$ with $J=(J_1,\dots,J_r)$ its corresponding index then 

\begin{equation*}
T_w^dY=
\left\{\begin{array}{l}
Z_{I_2,\dots,I_r,0}=0 \quad \forall \:(I_2,\dots,I_r)\\

Z_{I_2,\dots,I_r,1}=0 \quad \text{with} \quad d((I_2,\dots,I_r),J)>d
\end{array}\right.
\end{equation*}

\begin{equation*}
T_{\pi(w)}^{\frac{d-1}{2}}Z=
\left\{\begin{array}{l}
Z_{I_2,\dots,I_r,1}=0 \quad \forall \: (I_2,\dots,I_r)\\

Z_{I_2,\dots,I_r,0}=0 \quad \text{with} \quad d((I_2,\dots,I_r),J)>\frac{d-1}{2}
\end{array}\right.
\end{equation*}
Now a general hyperplane $H \supset T_w^dY$ in $H_Y$ has equation 
\begin{equation*}
H=
\left\{\begin{array}{l}
Z_{I_2,\dots,I_r,0}=0 \quad \forall\: (I_2,\dots,I_r)\\

\sum \alpha_{(I_2,\dots,I_r)}Z_{I_2,\dots,I_r,1}=0 \quad \textit{with} \quad d((I_2,\dots,I_r),J)>d
\end{array}\right.
\end{equation*}

Finally 

\begin{equation*}
\left\langle H, T_{\pi(w)}^{\frac{d-1}{2}}Z \right \rangle=
\left\{\begin{array}{l}
Z_{I_2,\dots,I_r,0}=0 \quad d((I_2,\dots,I_r),J)>\frac{d-1}{2}\\

\sum \alpha_{(I_2,\dots,I_r)}Z_{I_2,\dots,I_r,1}=0 \quad \text{with} \quad d((I_2,\dots,I_r),J)>d
\end{array}\right.
\end{equation*}

and since by construction we have that $$H_Z=\{Z_{I_2,\dots,I_r,1}=0 \quad \forall\: (I_2,\dots,I_r)\}$$ with $T_{\pi(w)}^{\frac{d-1}{2}}Z \cap Z=\pi(w)$ we conclude.
\end{proof}

\begin{Remark}\label{unb}
The previous corollary gives an asymptotic bound for the identifiability of $SV_{d}^{n}$ depending only on the values of $\textbf{n}=(1,n_2\dots,n_r)$ and $\textbf{d}=(1,d_2\dots,d_r)$. 

Note that in our case, i.e. for a Segre-Veronese in which there is a $\p^1$ factor embedded linearly, the bound on secant defectiveness given in \cite{AMR17} is trivial while the bound coming from Corollary \ref{CorSV} ensures that $SV_{d}^{n}$ is $h$-identifiable asymptotically for 
$$
h \sim n_2^{\lfloor log_2(d)\rfloor}$$
Finally, Theorem \ref{gen-twd} does not require a further numerical assumption involving the rank. Indeed, at the best of our knowledge, the principal result in order to prove identifiability is the one in \cite{CM19}, in which it is shown that the extra inequality $h \geq 2\dim(SV_{d}^{n})$ has to be fulfilled.
\end{Remark}

\bibliographystyle{amsalpha}
\bibliography{Biblio}
\end{document}